\documentclass{amsart}
\usepackage[ngerman,french,english]{babel} %Multi-language abstract
\usepackage{euler} %font style
\usepackage[utf8x]{inputenc}
\usepackage{amsmath, amssymb, graphics, setspace, scalerel}
\usepackage{fancyhdr}
\usepackage[margin=1in]{geometry}
\newcommand\ddfrac[2]{\frac{\displaystyle #1}{\displaystyle #2}}

\newcommand{\mathsym}[1]{{}}
\newcommand{\unicode}[1]{{}}
\usepackage{graphicx}
\usepackage[rightcaption]{sidecap}
\usepackage{amsthm}
\usepackage{hyperref}

\usepackage{tikz-cd}
\usetikzlibrary{cd}

\usepackage{xcolor}

\usepackage{mathtools}
\DeclarePairedDelimiter{\ceil}{\lceil}{\rceil}
\DeclarePairedDelimiter\floor{\lfloor}{\rfloor}

\usepackage{multicol}
\usepackage{scalerel,stackengine}
\stackMath
\newcommand\reallywidehat[1]{%
\savestack{\tmpbox}{\stretchto{%
  \scaleto{%
    \scalerel*[\widthof{\ensuremath{#1}}]{\kern.1pt\mathchar"0362\kern.1pt}%
    {\rule{0ex}{\textheight}}%WIDTH-LIMITED CIRCUMFLEX
  }{\textheight}% 
}{2.4ex}}%
\stackon[-6.9pt]{#1}{\tmpbox}%
}
\newtheorem{theorem}{Theorem}[section]
\theoremstyle{plain}

\newtheorem{corollary}[theorem]{Corollary}

\newtheorem{definition}[theorem]{Definition}
\newtheorem{example}[theorem]{Example}

\newtheorem{lemma}[theorem]{Lemma}

\newtheorem{proposition}[theorem]{Proposition}

\numberwithin{equation}{section}

\allowdisplaybreaks

\begin{document}
\title[The Atiyah-Bott-Lefschetz Formula Applied to the Based Loops on SU(2)]
{The Atiyah-Bott Lefschetz Formula Applied to the Based Loops on SU(2)}
\author{Jack Ding}

\maketitle

\begin{abstract}
The Atiyah-Bott-Lefschetz Formula is a well-known formula for computing the equivariant index of an elliptic operator on a compact smooth manifold. We provide an analogue of this formula for the based loop group $\Omega SU(2)$ with respect to the natural $(T \times S^1)$-action. From this result we also derive an effective formula for computing characters of certain Demazure modules.
\end{abstract} 

\selectlanguage{french} 
\begin{abstract}
La formule d'Atiyah-Bott-Lefschetz est une formule bien
connue pour l'indice \'equivariante d'un op\'erateur elliptique  sur 
une vari\'et\'e lisse compacte. Nous donnons une analogue de cette
formule pour le groupe de lacets bas\'es $\Omega SU(2)$ 
par rapport \`a l'action naturelle de $T \times S^1$. Avec ce
r\'esultat nous d\'emontrons aussi une formule effective pour
les caract\`eres de certaines modules de Demazure.
\end{abstract}

\selectlanguage{english} 
\section{Introduction}
Let $G$ be a compact, simply connected Lie group. Denote by $\Omega G$ the set of all smooth loops based at $e \in G$, i.e. $\Omega G = \{f \in C^{\infty}(S^1, G) \hspace{0.1cm}| \hspace{0.1cm}f(1) = e\}$. We then consider the subgroup of polynomial loops $\Omega_{poly} G = \{f \in \Omega G \hspace{0.1cm}| \hspace{0.1cm} f(t) = \displaystyle\sum\limits_{k = -N}^N a_k t^k , a_k \in M_{n \times n}(\mathbb{C})\}$, the based loops with finite Fourier expansion. Let $G_{\mathbb{C}}$ be the complexification of $G$. It is known that $\Omega_{poly} G$ is homotopy equivalent to the affine Grassmannian of $G_{\mathbb{C}}$ defined by $Gr = G_{\mathbb{C}}((t))/G_{\mathbb{C}}[[t]]$, where $G_{\mathbb{C}}((t))$ and $G_{\mathbb{C}}[[t]]$ denote Laurent series with values in $G_{\mathbb{C}}$ and power series with values in $G_{\mathbb{C}}$ respectively.

$\Omega G$ is an infinite-dimensional symplectic manifold equipped with the 2-form 
\begin{equation}
\omega(X(t), Y(t)) = \frac{1}{2\pi} \displaystyle\int\limits_0^{2\pi} \langle X(t), Y^{'}(t) \rangle \hspace{0.1cm} dt
\end{equation}
where $X(t), Y(t) \in T_e\Omega G \cong L\mathfrak{g}/\mathfrak{g} = \Omega \mathfrak{g}$, the space of loops in $\mathfrak{g}$ based at 0. For elements in the tangent space around an arbitrary point $T_p\Omega G$ the symplectic form is computed by pulling back the tangent vectors to the identity (via left or right multiplication), i.e. to elements of $\Omega \mathfrak{g}$ (see \cite{ps} chapter 8 for more details). \\
There are two natural Lie groups acting on $\Omega G$: First the maximal torus $T \subset G$ acts by pointwise conjugation $k \cdot \theta(t) = k\theta(t) k^{-1}$ for $k \in T$, second the circle group $S^1$ acts by rotating the loops. However, since the loops must be based, we instead have a "based rotation" given by $\theta \cdot \gamma(t) = \gamma(t + \theta) \gamma(\theta)^{-1}$ for $\theta \in S^1$. These two actions commute with each other. (See \cite{ap}, \cite{ps} for more details)\\
This $(T \times S^1)$-action is in fact Hamiltonian with respect to $\omega$, and the corresponding moment map is given by 
$\mu = (E, \rho): \Omega_G \rightarrow Lie(T \times S^1)^*$ with
\begin{align*}
E(\gamma) = \frac{1}{4\pi} \displaystyle\int\limits_0^{2\pi} ||\gamma(t)^{-1} \gamma^{'}(t)||^2 dt \\
\rho(\gamma) = \frac{1}{2\pi} \displaystyle\int\limits_0^{2\pi} pr_{Lie(T)}(\gamma(t)^{-1} \gamma^{'}(t)) dt
\end{align*}
The moment map and symplectic form restrict to the polynomial loops, and for the remainder of the paper we will refer to $\Omega_{poly} G$ as $\Omega G$ for simplicity.\\

In the 1960s, it was conjectured that the $G$-equivariant index of an elliptic operator can be computed purely in terms of data around the $G$-fixed points. Such a formula was proven by Atiyah and Bott \cite{ab} and is given in Section \ref{ab-formula}.
As we have mentioned previously, $\Omega G$ is homotopy equivalent to $Gr$, so we may also consider $\Omega G$ as an infinite dimensional Grassmannian. As such, it exhibits properties similar to those of flag manifolds and Grassmannians. One such property is the Bruhat decomposition, which says that $Gr$ is an (infinite) disjoint union of affine spaces called \textbf{Bruhat cells}. This property implies that $Gr$ has a filtration by finite-dimensional (possibly singular) spaces formed by finite unions of Bruhat cells, these spaces are known as \textbf{Schubert varieties}. \\
The Atiyah-Bott formula however, applies only to smooth manifolds. To this end, we use a well-known desingularization of Schubert varieties called \textbf{Bott-Samelson manifolds}. It was shown by Kumar \cite{kum} that:
\begin{enumerate}
\item The $(T \times S^1)$-equivariant indices of the prequantum line bundles on the Schubert varieties and Bott-Samelson manifolds are isomorphic
\item The equivariant index of the prequantum line bundle on $\Omega G \cong Gr$ is the inverse limit of the equivariant indices of the Schubert varieties
\end{enumerate}
Armed with these two facts, we may then compute positive-energy representations of $\Omega G$ in the case of $G = SU(2)$ as a limit of equivariant indices of Bott-Samelson manifolds, on which we may employ the Atiyah-Bott fixed point formula. Kumar proved that the character formula for the Bott-Samelson was the same as the character formula for the Schubert varieties. So we prove a character formula for the Bott-Samelson manifolds, and take a limit as the dimension of the Bott-Samelson manifolds goes to infinity. \\
As a corollary, we obtain an \textbf{effective} character formula for Demazure modules of the affine Lie algebra $\hat{\mathfrak{sl}}_2$ (i.e. a formula containing explicit weights). The character of Demazure modules was previously only known as a composition of Demazure operators. We also obtain an affine analogue of the Kostant multiplicity formula for semisimple Lie algebras.

\subsection{The Outline of the Paper}
In section 2 we introduce some basic facts about affine Weyl groups. In section 3 we review a topological construction of Bott-Samelson manifolds and certain affine coordinates on them. In sections 4-5 we review prequantization of (pre-)symplectic manifolds and discuss the relations to the Atiyah-Bott formula.
In section 7 we review the Kostant-Kumar nil-Hecke algebra and formulate our main theorem in two versions: Theorem \ref{v1} and Theorem \ref{v2}. These theorems describe analogues of the Atiyah-Bott formula for $\Omega SU(2)$. In section 8 we apply our theorems (and their proofs) to derive characters of Demazure modules for $\hat{\mathfrak{sl}_2}$. Finally we read off our character formula to derive a Kostant multiplicity formula for the affine Lie algebra.

\section{The Affine Weyl Group} \label{afweyl}
In this section we follow the notes of Magyar \cite{mag}.
First we recall that the Weyl group associated to type~$A_{N-1}$ is isomorphic to the symmetric group $S_N$. Given the finite Weyl group $W$ associated to type~$A_{N-1}$, the affine Weyl group $\overline{W}$ is the semidirect prodct $W \ltimes Q$, where $Q$ denotes the coroot lattice of $SU(N)$. Another way of expressing this is that each element in $\overline{W}$ is given by the product of a reflection $r_{\alpha}$ and a translation~$t^{\beta}$, where $\alpha$ is a root and $\beta \in Q$. \\

The generators include the standard generators for the finite Weyl group: $s_i = r_{\alpha_i}$ for $i = 1,\cdots, N - 1$. There is an extra generator in the affine Weyl group given by $s_0 = r_{\theta}t^{-\theta^{\vee}}$ where $\theta$ is the highest root. For these generators, the multiplication is given by 
$$(r_{\alpha_{i_1}}t^{\beta_{i_1}})(r_{\alpha_{i_2}}t^{\beta_{i_2}}) = r_{\alpha_{i_1}}r_{\alpha_{i_2}} t^{\beta_{i_1} + r_{\alpha_{i_1}}(\beta_{i_2})}.$$
        
\begin{example}
The group $SU(2)$ has a coroot lattice isomorphic to $\mathbb{Z}$ and its Weyl group is the group $\mathbb{Z}_2$ with generator $r$, so the Weyl group for affine $\mathfrak{sl}_2$ is the semidirect product $\mathbb{Z}_2 \ltimes \mathbb{Z}$. 
\end{example}

\section{Bott-Samelson Manifolds}
First we state the important Bruhat decomposition. Let $\mathcal{B}$ denote the Iwahori subgroup $\{f \in G_{\mathbb{C}}[[t]] \hspace{0.1cm}| \hspace{0.1cm} f(0) \in B\}$ where $B$ is a the standard Borel subgroup of $G$. For $G = SU(N)$, $B$ is given by the group of upper triangular matrices. For general $G$ we can embed $G$ into $SU(N)$ for some $n$ and then take upper triangular matrices intersected with the image of $G$. Let $\ell = rank \hspace{0.1cm}G + 1$, then we have the following definition

\begin{definition}[Kumar, \cite{kum} Definition 6.1.18, p. 186]
For any subset $Y \subset \{1, \cdots, \ell\}$ define the \textbf{standard parabolic subgroup} $\mathcal{P}_Y$ as $\mathcal{P}_Y = \mathcal{B} W_Y \mathcal{B}$ where $W_Y$ is the subgroup of the affine Weyl group generated by~$\{s_i\}_{i \in Y}$. 
\end{definition}

We consider the case where $Y = \{1, \cdots, \ell\}$ and we denote the corresponding subgroup $\mathcal{P}_Y$ simply by $\mathcal{P}$. Of course in this case, the group $W_Y$ is simply the finite Weyl group $W$.

\begin{theorem}[Bruhat Decomposition]
$Gr = \sqcup_{w \in \overline{W}} \mathcal{B}w\mathcal{P}/\mathcal{P}$, each component $\mathcal{B}w\mathcal{P}/\mathcal{P}$ is isomorphic to an affine space of dimension $\ell(w)$.
\end{theorem}

For any element $x$ in a poset $S$, one may consider its lower order ideal defined by $I_x = \{y \in S \hspace{0.1cm} | \hspace{0.1cm} y \le x\}$. The affine Weyl group is a poset with order given by the Bruhat order. Therefore for any Weyl group element~$w \in \overline{W}$, we can take its lower order ideal $I_w := \{v \in \overline{W} \hspace{0.1cm}| \hspace{0.1cm} v \le w\}$. Then the disjoint union $\sqcup_{v \in I_w} \mathcal{B}v\mathcal{P}/\mathcal{P}$ is known as the Schubert variety associated to $w$. It is denoted by $X_w$. \\
In general $X_w$ may have singularities. In fact in the case of the finite-dimensional flag manifold $G_{\mathbb{C}}/B$, $X_w$ is smooth if and only if $w$, seen as an element of $S_n$, avoids $(1324)$ and $(2143)$ (see \cite{ulfarsson} for example). Therefore in order to use Atiyah and Bott's formula, we must use the desingularization given by Bott-Samelson manifolds. \\

In this paper we will use Magyar's topological description of Bott-Samelson manifolds (as opposed to the description using lattices in $Gr$). This approach allows us to use well-known coordinates on the Bott-Samelson manifolds to compute fixed point data. \\

A reduced word is one without redundant pairs (i.e. expressions like $xx^{-1}$). For each reduced word $w$, denote by $M_w$ the Bott-Samelson associated to $w$. It is a desingularization of~$X_w$.

\subsection{Magyar's Topological Description}
For this paper it suffices to give a description in type A. Let $\alpha = e_i - e_j$ be a root and $k$ be an integer. 
Define 
$K_{(e_i - e_j, k)}$ as the group generated by the maximal torus $T \in K$ and matrices with a copy of $SU(2)$ in the (ij)-block, in the form 
\[
\begin{pmatrix}
	a & -\overline{b}t^k \\
	bt^{-k} & \overline{a}
	\end{pmatrix}
	\]	
In type $A$, the root $e_i - e_{i+1}$ corresponds to the simple root $\alpha_i$ for $i = 1, \cdots, n-1$. \\
Let $\alpha_0 = (-\displaystyle\sum\limits_{i}^{n-1} \alpha_i, 1)$ and $\alpha_i = (\alpha_i, 0)$ for $i \ge 1$ denote the affine simple root and the finite simple roots respectively.
We now denote by $K_{(i)}$ the group $K_{\alpha_i}$. These groups $K_{\alpha_i}$ are subgroups of the free loop group $LK$. 

\begin{definition}
Let $w = s_{i_1} s_{i_2} \cdots s_{i_n}$ be a reduced word. The Bott-Samelson manifold associated to $w$ is given by the group $M_w := \left(K_{(i_1)} \times K_{(i_2)} \times \cdots \times K_{(i_n)} \right)/T^n$ where 
$K_{(i_1)} \times K_{(i_2)} \times \cdots \times K_{(i_n)}$ has the following action of $T^n$:
$$ (x_1, \cdots, x_n) \cdot (t_1, \cdots, t_n) = (x_1t_1, t_1^{-1}x_2t_2, \cdots, t_{n-1}^{-1}x_nt_n)$$
\end{definition}

We also give a topological description for the affine Schubert varieties living in $\Omega G$:
\begin{definition}
Let $w = s_{i_1} s_{i_2} \cdots s_{i_n}$ be a reduced word. Consider the group $K_w := K_{(i_1)} \cdot K_{(i_2)} \cdot \hdots \cdot K_{(i_n)}$ which lives in $LG$. Then the Schubert variety is defined to be the image of $K_w$ under the basepoint map $b: LG \rightarrow \Omega G$ given by $f(t) \mapsto f(t) \cdot f(1)^{-1}$. 
\end{definition}

\subsection{Coordinates on $M_w$}
For a Bott-Samelson manifold $M_w$ there is a birational multiplication map $\eta: M_w \rightarrow X_w$ to the corresponding Schubert variety. This map is given by 
$$\eta(x_1(t), \cdots, x_n(t)) =  x_1(t)\cdot \hdots \cdot x_n(t) \left(x_1(1) \cdot \hdots \cdot x_n(1) \right)^{-1} $$

The Schubert variety embeds into the based loop group and it inherits the $(T \times S^1)$-action, where $T$ acts by conjugation and $S^1$ acts by rotation. \\
We want an action on $M_w$ that is equivariant with respect to $\eta$. The action that does the job is 
\begin{equation}
t \cdot (k_1, \cdots, k_n) = (tk_1, k_2, \cdots, k_{n-1}, k_n)
\end{equation}
\begin{equation}
\theta \cdot (k_1, \cdots, k_n) = (\theta \cdot k_1, \theta \cdot k_2, \cdots, \theta \cdot k_{n-1}, \theta \cdot k_n)
\end{equation}
We observe that $M_w$ is defined by the equivalence relation
\begin{equation}
(k_1, \cdots, k_n) = (k_1t_1, t_1^{-1}k_2t_2, \cdots, t_{n-1}^{-1}k_n t_n)
\end{equation}
Hence the fixed points are $(k_1, \cdots, k_n)$ such that $k_j = \dot{s}_{i_j}$ or $1$ .
We define $\dot{w}$ as follows:
If $w = s_i$ for some $i \in \{1, \cdots, n-1\}$, then let $\dot{w}$ denote a lift of an element of the Weyl group $N_G(T)/T$ to $N_G(T)$. An explicit formula for a canonical lift of $s_{\alpha}$ is given by
$$ \exp(F_\alpha) \exp(-E_\alpha) \exp(F_\alpha)$$
where $E_\alpha, F_\alpha, H_\alpha$ is the $\mathfrak{sl}_2$-triple associated to the root $\alpha$. \\
Therefore we have a total of $2^n$ isolated fixed points.\\

Let $G$ be a compact Lie group, for any simple root $\alpha \in \mathfrak{g}^*$ there is an $SU(2)$-subgroup associated to $\alpha$. Let $\Psi_\alpha$ denote the embedding of that $SU(2)$-subgroup into $G$. Then we have the following description of the affine coordinates on the Bott-Samelson manifold (see \cite{lu}):

\begin{proposition} 
	There are affine charts around each fixed point $k = (k_1, \cdots, k_n) \in M_w$ given by 
	\begin{equation}
	\Phi_{(k_1, \cdots, k_n)}: (z_1, \cdots, z_n) \mapsto (\Psi_{\alpha_{i_1}}(m_{z_1}) \cdot k_1, \cdots, \Psi_{\alpha_{i_n}}(m_{z_n}) \cdot k_n)
	\end{equation}
	where 
	\[m_{z_i} := \frac{1}{\sqrt{1 + |z_i|^2}}
	\begin{pmatrix}
	1 & -\bar{z_i} \\
	z_i & 1
	\end{pmatrix}
	\]	
	if $k_i = e$ and 
		\[m_{z_i} := \frac{i}{\sqrt{1 + |z_i|^2}}
	\begin{pmatrix}
	z_i & 1 \\
	1 & -\bar{z_i}
	\end{pmatrix}
	\]	
	if $k_i = \dot{s_i}$.
\end{proposition}
Even though these coordinates are not holomorphic, we can still treat $M_w$ as a real manifold of dimension $2n$ and use the real coordinates $x_i, y_i$ for each $z_i$. \\

Each Bott Samelson manifold $M_w$ has $2^{\ell(w)}$ fixed points, indexed by what are called \textbf{galleries}. We now give the following definition due to H\"arterich \cite{harterich}: 

\begin{definition}
Let $w = s_{i_1} \cdots s_{i_n}$ be a reduced word. A \textbf{gallery} associated to $w$ is an element in the set $\{0, 1\}^n$. We denote a gallery $\gamma$ by an $n$-tuple $(\gamma_1, \cdots, \gamma_n)$. We define the realization of a gallery $\gamma$ by the word $s_{i_1}^{\gamma_1} \cdots s_{i_n}^{\gamma_n}$; it is denoted by $u(\gamma)$.
\end{definition}

\begin{example}
Let $w = s_1s_2s_1$, then there are $2^3 = 8$ possible galleries associated to $w$. If we take the gallery $\gamma = (0, 1, 1)$, then its realization is $u(\gamma) = s_1^0 s_2^1 s_1^1 = s_2s_1$. Note that two galleries may have the same realization, e.g. $u( (0, 0, 0))$ and $u((1, 0, 1))$ are both equal to the identity word $e$.
\end{example}

\begin{proposition}
Let $w$ be a reduced word. The fixed points of $M_w$ under the $(T \times S^1)$-action are in bijection with the set of galleries associated to $w$. In particular $M_w$ has $2^{\ell(w)}$ fixed points. 
\end{proposition}

\subsection{Euler class of tangent spaces at fixed points}
In the general situation where a torus $T$ acts on a manifold $M$, suppose that the $T$-action has isolated fixed points. Then for every $p \in M^T$, there is an induced $T$-action on the tangent space $T_p M$. Furthermore of $M$ is even dimensional, then $T_p M$ splits into a direct sum of copies of $\bigoplus\limits_{i=1}^n \mathbb{C}_i$, and $T$ acts on each $\mathbb{C}_i$ by a weight $\beta_i$. This is what is known as the \textbf{isotropy representation} of $T$. \\

With the coordinates on $M_w$ from the previous section, we can then compute the weights of the isotropy representation of $T \times S^1$ on the fixed points of $M_w$. \\

At the fixed point $k = (k_1, \cdots, k_n)$, we notice that the affine chart sends $(z_1, \cdots, z_n)$ to $(\Psi_{\alpha_{i_1}}(m_{z_1}) \cdot k_1, \cdots, \Psi_{\alpha_{i_n}}(m_{z_n}) \cdot k_n)$. By the definition of Bott-Samelson manifolds, we have
\begin{align*}
&(\Psi_{\alpha_{i_1}}(m_{z_1}) \cdot k_1, \Psi_{\alpha_{i_1}}(m_{z_2}) \cdot k_2, \cdots, \Psi_{\alpha_{i_n}}(m_{z_n}) \cdot k_n) \\
= \hspace{0.2cm}&(\Psi_{\alpha_{i_1}}(m_{z_1}), k_1 \Psi_{\alpha_{i_1}}(m_{z_2}) \cdot k_2,\cdots, \Psi_{\alpha_{i_n}}(m_{z_n}) \cdot k_n) \\
= \hspace{0.2cm}&(\Psi_{\alpha_{i_1}}(m_{z_1}), k_1 \Psi_{\alpha_{i_1}}(m_{z_2}) k_1^{-1},\cdots, (k_1k_2 \cdots k_{n-1)}\Psi_{\alpha_{i_n}}(m_{z_n}) (k_1k_2 \cdots k_{n-1})^{-1} \cdot k_n) \\
\intertext{(here we multiplied $k_2^{-1}k_1^{-1}$ to the end of the second component and multiplied $k_1k_2$ at the beginning of the third component, and so on...)}
= \hspace{0.2cm}&(\Psi_{\alpha_{i_1}}(m_{z_1}), k_1 \Psi_{\alpha_{i_1}}(m_{z_2}) k_1^{-1},\cdots, (k_1k_2 \cdots k_{n-1)}\Psi_{\alpha_{i_n}}(m_{z_n}) (k_1k_2 \cdots k_{n-1})^{-1}) \\
\intertext{(in this final step we multiplied $k_n^{-1}$ to the end of the last component)}
\end{align*}

If one of the $k_i$ is $id$, then it does not change the $\Psi_{\alpha_{i_n}}(m_{z_n})$ factor. However if $k_i = \cdot{s_i}$ then the conjugation $k_i \Psi_{\alpha_{i_n}}(m_{z_n})k_i^{-1}$ will have the effect of the simple reflection $s_{\alpha_i}$. Therefore $(T \times S^1)$ acts by 
\begin{equation}
(t, \theta) \cdot (z_1, \cdots, z_n) = (t^{-k_1(\alpha_{i_1})} \cdots z_1, t^{-k_1k_2(\alpha_{i_2})}, \cdots, t^{-k_1k_2\cdots k_n(\alpha_{i_n})} )
\end{equation}
\\
At the fixed point $k = (k_1, \cdots, k_n)$, we differentiate the affine chart $\Phi_{(k_1, \cdots, k_n)}$ at 0. First we compute $\frac{\partial}{\partial z_i} m_{z_i}$ for each $z_i$. If $k_i = e$ then 
\begin{align*}
\frac{\partial}{\partial z_i} m_{z_i} = &
	\frac{\partial}{\partial z_i} \frac{1}{\sqrt{1 + |z_i|^2}}
	\begin{pmatrix}
	1 & -\bar{z_i} \\
	z_i & 1
	\end{pmatrix}
	\\
	= &	
	\begin{pmatrix}
	\frac{-\bar{z_i}}{2(1 + |z_i|^2)^{\frac{3}{2}}} & \frac{\bar{z_i}^2}{2(1 + |z_i|^2)^{\frac{3}{2}}}\\
	\frac{-\frac{|z_i|^2}{2\sqrt{1 + |z_i|^2}} + 1}{(1 + |z_i|^2)} & \frac{-\bar{z_i}}{2(1 + |z_i|^2)^{\frac{3}{2}}}
	\end{pmatrix}	
\end{align*}
Evaluating at $z_i = 0$ we get 
$\begin{pmatrix}
	0 & 0 \\
	1 & 0
	\end{pmatrix}$. 
Likewise if we take $\frac{\partial}{\partial \bar{z_i}} m_{z_i}|_{z_i = 0}m_{z_i} = $ then we get the matrix $\begin{pmatrix}
	0 & -1 \\
	0 & 0
	\end{pmatrix}$ 
	
Since each $\Psi_{\alpha}$ is an embedding into the $SU(2)$-subgroup of $G$ corresponding to the root $\alpha$. The matrices computed above get sent via $T\Psi$ to $F_{\alpha}$ and $-E_{\alpha}$ respectively. \\
Now suppose $k_i = \dot{s_i}$ then since $s_i$ is the simple reflection corresponding to the simple root $\alpha_i$. This means that the tangent vectors at $(\cdots, \dot{s_i}, \cdots)$ are $(T\Psi_{\alpha_i})_{k_i} \frac{\partial}{\partial z_i} m_{z_i}$ and $(T\Psi_{\alpha_i})_{k_i} \frac{\partial}{\partial \overline{z_i}} m_{z_i}$. By a similar calculation as above, we get that $\frac{\partial}{\partial z_i} m_{z_i}|_{z_i = 0} = \begin{pmatrix}
	1 & 0 \\
	0 & 0
	\end{pmatrix}$  and 
$\frac{\partial}{\partial \overline{z_i}}|_{z_i = 0} = 
\begin{pmatrix}
	0 & 0 \\
	0 & -1
	\end{pmatrix}$.
Then we see that $T\Psi(\begin{pmatrix}
	1 & 0 \\
	0 & 0
	\end{pmatrix}) \cdot \dot{s_i}$ equals exactly $E_{\alpha_i}$ and 
	$T\Psi(\begin{pmatrix}
	0 & 0 \\
	0 & -1
	\end{pmatrix}) \cdot \dot{s_i}$ equals exactly $-F_{\alpha_i}$.

Now taking into account the $(k_1k_2 \cdots k_{n-1})(\hdots) (k_1k_2 \cdots k_{n-1})^{-1}$ conjugation action on the $n$-th component, we get the following corollary.

\begin{corollary}
	The isotropy weights of the tangent space of $M_w$ at the fixed point $k = (k_1, \cdots, k_n)$ are given by $\left\{s_{i_1}^{e_1} \cdot (-\alpha_{i_j}), \hdots,   \left( \prod\limits_{k=0}^n s_{i_k}^{e_k}\right) \cdot( -\alpha_{i_n})\right\}$, where $e_j = 1$ if $k_j = \dot{s}_{i_j}$ and $e_j = 0$ if $k_j = \dot{1}$. \\
	The equivariant Euler class of the tangent space of $M_w$ at the fixed point $k = (k_1, \cdots, k_n)$ is given by $\left( \prod\limits_{j=1}^n (\prod\limits_{k=1}^j s_{i_k}^{e_k}) \cdot( -\alpha_{i_j}) \right)$. 
\end{corollary}

\section{Geometric Quantization}

\subsection{Prequantum line bundles}\label{preq}

A prequantum line bundle over a symplectic manifold $(M, \omega)$ is a Hermitian line bundle with connection $(L, \langle, \rangle, \nabla)$, such that the curvature $\Omega$ of the connection is cohomologous to $\omega$. Let $\Phi$ denote the moment map of a Hamiltonian $G$-action on $M$, then for all $X \in \mathfrak{g}$ we define the quantum operator $H_X := \nabla_{X^{\#}} + i\langle \Phi, X \rangle$, where $X^{\#}$ is the fundamental vector field associated to $X$. \\
The map $X \mapsto H_X$ is a Lie algebra homomorphism, and each $H_X$ acts on the space of sections of $L$. Therefore we can consider the sections of $L$ as a representation of $G$ (through the exponential map). Note that if $p \in M$ is a fixed point, then $\nabla_{X^{\#}} = 0$ for all $X \in \mathfrak{g}$. So the character of $g = \exp(X)$ is $\exp(i\langle \Phi, X\rangle)$. \\
If in addition $M$ is Kahler, then we can use the Kahler polarization on $M$ to get a representation of $G$ on the space of holomorphic sections. 

\subsection{Line bundles on $\Omega G$}

Let $\mathcal{L}(\lambda)$ denote the prequantum line bundle on $\Omega G$ associated to the dominant weight $\lambda$. This bundle can be restricted to Schubert varieties $X_w$ and then pulled back to the Bott-Samelson manifold $M_w$ to get a prequantum line bundle $\eta^*\mathcal{L}(\lambda)|_{X_w}$ on $M_w$. 

\begin{proposition}[Kumar]
The line bundles on $\Omega G$ are parametrized by $k\overline{\omega_0}, k \in \mathbb{Z}$ where $\overline{\omega_0}$ is the 0-th fundamental weight of $\hat{\mathfrak{g}}$. Therefore $Pic(\Omega G) \cong \mathbb{Z}$. 
\end{proposition}

\section{The Atiyah-Bott fixed point theorem} 

In this section we state a special case of the Atiyah-Bott Lefschetz formula for equivariant indices of elliptic operators. Let $T$ be a torus and let $M$ be an almost-complex $T$-manifold equipped with a $T$-equivariant vector bundle $V$. Then the operator which applies to our case is the Dirac-Dolbeault operator $\overline{\partial}_V + \overline{\partial}^{*}_V$. The index of this operator is given by the alternating sum of sheaf cohomologies $\displaystyle\sum\limits_i (-1)^i tr (H^i(M, V))$ where $T$ has a natural action on each cohomology group $H_i(M, V)$. \\
Furthermore supposed that $p \in M^T$ is a fixed point, then there is a natural action of $T$ on $T_p M$ (see Section~3.3). This gives rise to a dual $T$-action on the cotangent space $T_p^*M$. Likewise since $V$ is $T$-equivariant, we get a $T$-action on the fiber $V_p$ as well.  
\label{ab-formula}
Now we state a very important formula by Atiyah and Bott \cite{ab}:
\begin{theorem}
	Let $T$ be a torus. Assume that $M$ is a manifold with a Hamiltonian $T$-action and that the fixed point set $M^T$ is finite. Then for any $T$-equivariant vector  bundle $V$ on $M$, we get that for all $t \in T$:
	\begin{equation}
	\displaystyle\sum\limits_i (-1)^i tr (t|_{H^i(M, V)}) = \displaystyle\sum\limits_{p \in M^T} \frac{tr (t|_{V_p})}{\det(Id - t|_{T^*_p M})}
	\end{equation}
\end{theorem}

Our goal is to derive a formula for the based loop group $\Omega SU(2)$ in the following form:
\begin{equation}\displaystyle\sum\limits_i (-1)^i tr ((t, \theta)|_{H^i(\Omega SU(2), V)}) = \displaystyle\sum\limits_{w \in \overline{W}} tr ((t, \theta)|_{V_w}) \cdot R_w(\Omega SU(2))
\end{equation} 
for all $(t, \theta) \in T \times S^1$. \\

\textbf{Remark:} From this point on we will abuse notation slightly and suppress the torus elements in the computations. I.e. $tr ((t, \theta)|_{V_w})$ will simply be denoted as $tr (V_w)$. Equations containing such terms will be understood to hold for all elements of the torus. \\

Here $R_w(\Omega SU(2)): T \times S^1 \rightarrow \mathbb{C}$ are certain rational functions in $T \times S^1$ which live in Kostant and Kumar's nil-Hecke ring (see Section 7). Since the fixed points of the $(T \times S^1)$-action are indexed by the affine Weyl group elements, this formula is in the exact same form as the Atiyah-Bott formula, except that it contains an infinite sum on the RHS. Note that written in this form, the analogous rational function for compact smooth $G$-manifolds would be $R_{p}(M) = \det(1 - T^*_p M)^{-1}$. 
In this case the character of the tangent cone $ch (gr(\mathcal{O}_{p, M}))$ is exactly equal to $\det(1 - T^*_p M)^{-1} = R_{p}(M)$. For a singular space however, the numerator of the character of the tangent cone will not necessarily be 1. \\

We run into a small technical difficulty: the based loop group has a filtration by Schubert varieties, which are in general very singular. Therefore we cannot apply this theorem directly. However, we can apply this result to the Bott-Samelson manifolds, which are desingularizations of Schubert varieties. We are allowed to do so due to the following result:

\begin{proposition}[Grossberg-Karshon, Kumar]  \cite{kum}, \cite{gk} \label{isom}
$H^p(M_w, \eta^*\mathcal{L}(\lambda)|_{X_w}) \cong H^p(X_w, \mathcal{L}(\lambda)|_{X_w})$
\end{proposition}

Furthermore, Kumar proved that
\label{invlimit}
\begin{theorem}[\cite{kum} Theorem 8.1.25 p. 271]  
$\varprojlim\limits_{w \in \overline{W}} H^0(M_w, \eta^*\mathcal{L}(\lambda)|_{X_w})^* \cong L_{\lambda}$
where $L_\lambda$ is the irreducible representation of the Kac-Moody group with highest weight $\lambda$. Here elements of $\overline{W}$ are ordered by the Bruhat order.
\end{theorem}

Therefore it suffices to apply the Atiyah-Bott fixed point theorem to the Bott-Samelson manifolds to obtain the character of the virtual line bundle $\displaystyle\sum\limits_{i} (-1)^i H^i(\mathcal{L}(\lambda)|_{X_w})$. Then we will take a limit over words $w$ in the affine Weyl group to obtain $\displaystyle\sum\limits_{i} (-1)^i H^i(\Omega SU(2), \mathcal{L}(\lambda))$ in the following sections. 

\subsection{Application to the Bott-Samelson Manifolds}
To summarize our setup so far, we now have:
\begin{itemize}
\item A smooth presymplectic manifold $(M_w, \eta^*(\omega|_{X_w}))$
\item A Hamiltonian $(T \times S^1)$-action on $M_w$ with moment map $\mu \circ \eta: M_w \rightarrow Lie(T \times S^1)^*$, and $2^n$ fixed points parametrized by \textbf{galleries} $\{\gamma \in \Gamma_w\}$
\item A prequantum line bundle $\eta^*\mathcal{L}(\lambda)|_{X_w}$ on $M_w$
\end{itemize}

Applying the Atiyah-Bott formula yields
\begin{align*}
\displaystyle\sum\limits_i (-1)^i tr (H^i(M_w, \eta^*\mathcal{L}(\lambda)|_{X_w})) 
&= \displaystyle\sum\limits_{p \in M_w^{T \times S^1}} \frac{tr (\eta^*\mathcal{L}(\lambda)|_{X_w}|_p)}{\det(1 - T^*_p M_w)}   \\
&= \displaystyle\sum\limits_{\gamma \in \Gamma_w} \frac{tr (\eta^*\mathcal{L}(\lambda)|_{X_w}|_{\gamma})}{\det(1 - T^*_{\gamma} M_w)} \text{\hspace{0.5cm}(the fixed points are galleries associated to $w$)} \\
&= \displaystyle\sum\limits_{\gamma \in \Gamma_w} \frac{e^{-i\mu(\gamma)}}{\det(1 - T^*_{\gamma} M_w)} \text{\hspace{0.5cm}(by Section \ref{preq})}\\
&= \displaystyle\sum\limits_{v \le w} e^{i(v \cdot \lambda)} \displaystyle\sum\limits_{\gamma: u(\gamma) = v}\frac{1}{\det(1 - T^*_{\gamma} M_w)} \\
&= \displaystyle\sum\limits_{v \le w} e^{i(v \cdot \lambda)} \displaystyle\sum\limits_{\gamma: u(\gamma) = v}\frac{1}{(1 - e^{s_{i_1}^{\gamma_1} \alpha_{i_1}}) \cdots (1 - e^{s_{i_1}^{\gamma_1} \cdots s_{i_n}^{\gamma_n}\alpha_{i_n}})}.
\end{align*}\label{char}
Here the last equality is due to the fact that $\det(1 - T^*_{\gamma} M_w) = \displaystyle\prod\limits_{\alpha} (1 - e^{-\alpha})$, and $\alpha$ ranges over all isotropy weights of $T_p^* M_w$. These were computed in the section on Bott-Samelson manifolds. \\
Note that by Proposition \ref{isom}, the last line is also equal to $\displaystyle\sum\limits_i (-1)^i tr (H^i(X_w, \mathcal{L}(\lambda)|_{X_w}))$.

\section{The localization theorem in equivariant K-theory}

The following two subsections are due to Kumar. (See \cite{kum} Chapters 11 and 12)

\subsection{Kostant-Kumar nil-Hecke ring} 
\label{nil-Hecke}
Let $\mathcal{G}$ be a Kac-Moody group and $\mathcal{T}$ be a maximal torus. Then denote by $A[\mathcal{T}]$ the group algebra of the character group of $T$. Denote by $Q[\mathcal{T}]$ the fraction field of $A[T]$ and let $\tilde{Q}[\mathcal{T}]$ be the completion of $A[\mathcal{T}]$: the set of formal infinite sums $\displaystyle\sum\limits_{e^{\lambda} \in X(\mathcal{T})} n_{\lambda}e^{\lambda}$ ($n_{\lambda} \in \mathbb{Z}$) with coefficients in $Q[\mathcal{T}]$. Then form the $Q[\mathcal{T}]$ -vector space $Q[\mathcal{T}]_{\overline{W}}$ with basis $\{\delta_w\}_{w \in \overline{W}}$. The space $Q[\mathcal{T}]_{\overline{W}}$ is an associative ring with multiplication given by

\begin{equation}
\left(\displaystyle\sum\limits_v q_v \delta_v \right)\cdot \left(\displaystyle\sum\limits_w q_w \delta_w \right) = 
\displaystyle\sum\limits_w q_v(vq_w) \delta_{vw} 
\end{equation}
and identity $\delta_e$ (see \cite{kum}, p. 450). 

\subsection{A special element}
Let $s_i$ be an element in $\overline{W}$ which is a simple reflection. Then there is a special element $T_{s_i} = \displaystyle\frac{1}{1 - e^{\alpha_i}} \delta_{s_i} + \frac{1}{1 - e^{-\alpha_i}} \delta_{e} \in Q[\mathcal{T}]_{\overline{W}}$. For a general word $w \in \overline{W}$, we define $T_w = T_{s_{i_1}} \cdots T_{s_{i_n}}$. Then we can calculate the coefficients of $T_w$ with respect to the $Q[\mathcal{T}]$-basis. Write $T_w = \displaystyle\sum\limits_{v \le w} b_{w, v} \delta_v$, then we have
\begin{equation}
b_{w, v} = \displaystyle\sum \left( (1 - e^{-s_{i_1}^{\varepsilon_1}\alpha_{i_1}}) \cdots (1 - e^{-s_{i_1}^{\varepsilon_1} \cdots s_{i_n}^{\varepsilon_n}\alpha_{i_n}}) \right)^{-1}
\end{equation}\label{bwv}
where the sum is over all n-tuples $(\varepsilon_1, \cdots, \varepsilon_n) \in \{0, 1\}^n$ satisfying $s_{i_1}^{\varepsilon_1} \cdots s_{i_n}^{\varepsilon_n} = v$. 
\begin{example}
	Let $\mathcal{G}$ be the Kac-Moody group associated to the affine Lie algebra $\hat{\mathfrak{sl}_2}$. We would like to compute $b_{s_1s_0, s_1}$. The only subword of $s_1s_0$ equalling $s_1$ is $s_1^1s_0^0$. So $b_{s_1s_0, s_1} = 
	\displaystyle\frac{1}{1 - e^{-s_1\alpha_1}} \cdot 
	\displaystyle\frac{1}{1 - e^{-s_1\alpha_0}}
	= 
	\displaystyle\frac{1}{1 - e^{\alpha_1}} \cdot 
	\displaystyle\frac{1}{1 - e^{-(\alpha_0 + 2\alpha_1)}}$
\end{example}
We also define an involution on $Q[\mathcal{T}]_{\overline{W}}$ by $\overline{e^{\lambda}} := e^{-\lambda}$. 

\subsection{Character of the tangent cone}
Let $Y = \mathcal{G} / \mathcal{B}$ denote the affine flag variety. Its fixed points under the $T$ action by left multiplication are parametrized by $v \in \overline{W}$. Kumar proved the following statement about the $T$-characters of tangent cones in the Schubert varieties $X_w$ around these points. 
\begin{theorem} [Kumar, \cite{kum} Theorem 12.1.3, p. 451] 
	\begin{equation}
	ch (gr(\mathcal{O}_{v, X_w})) = \overline{b_{w, v}}
	\end{equation}
\end{theorem}

However we are working with the case of the affine Grassmannian $Gr = \mathcal{G} / \mathcal{P}$ where $\mathcal{P}$ is the maximal parabolic containing all the finite simple roots. Here the $T$-fixed points are parametrized by minimal length cosets in $\overline{W} / W$. So we make a slight modification to the formula above.

Recall the application of the Atiyah-Bott fixed point formula to the Bott-Samelson manifold $M_w$ given in section (\ref{char}), the last line contained expressions which may be written in the form of $b_{w, v}$ for certain words~$w, v$. Indeed we may rewrite it as

\begin{align*}
&\displaystyle\sum\limits_{v \le w} e^{i(v \cdot \lambda)} \displaystyle\sum\limits_{\gamma; u(\gamma) = v}\frac{1}{(1 - e^{s_{i_1}^{\gamma_1} \alpha_{i_1}}) \cdots (1 - e^{s_{i_1}^{\gamma_1} \cdots s_{i_n}^{\gamma_n}\alpha_{i_n}})} \\
= & \displaystyle\sum\limits_{v \le w} e^{i(v \cdot \lambda)} \cdot \overline{b_{w, v}} \\
= & \displaystyle\sum\limits_{\text{cosets   } vW, v \le w} e^{i(v \cdot \lambda)} \displaystyle\sum\limits_{v^{'} \in \overline{W}, vW = v^{'}W} \overline{b_{w, v^{'}}}.
\end{align*}
The last equality follows because the fixed points in the Schubert variety are indexed by cosets $vW$ and galleries whose realizations belong to the same cosets are mapped to the same value under $\mu$. \\
By virtue of Proposition \ref{isom}, we now have a fixed point formula for the Schubert varieties as well:

\begin{proposition}  
	\begin{equation}
	\displaystyle\sum\limits_i (-1)^i tr (H^i(X_w, \mathcal{L}(\lambda)|_{X_w})) = \displaystyle\sum\limits_{\text{cosets   } vW, v \le w} e^{i(v \cdot \lambda)} \displaystyle\sum\limits_{v^{'} \in \overline{W}, vW = v^{'}W} \overline{b_{w, v^{'}}}
	\end{equation}
\end{proposition}
In the language of Section 4, $R_v(X_w) = \displaystyle\sum\limits_{v^{'} \in \overline{W}, vW = v^{'}W} \overline{b_{w, v^{'}}}$. In the next section we will derive an effective formula for $R_v(X_w)$ for the Schubert varieties of $\Omega SU(2)$.

\section{Calculations for G = SU(2)}

We can consider the space of polynomial based loops in $SU(2)$ as the affine Grassmannian $\mathcal{G}/\mathcal{P}$ where~$\mathcal{P}$ is the maximal parabolic subgroup corresponding to the set $I = \{1\}$. For the reduced Weyl group~$\overline{W} / W_{\{1\}}$ (see Section \ref{afweyl} for definitions), it is totally ordered w.r.t Bruhat order, with an ascending chain $s_0 \le s_1 s_0 \le s_0 s_1 s_0 \le ... $ (see Section 2). Then we denote the length $n$ element of this chain by $w_n$. In this way, the $(T \times S^1)$-fixed points of $\Omega SU(2)$ are indexed by the integers. \\

Moreover $\Omega SU(2)$ has a filtration by Schubert varieties, and in this case they correspond exactly to the elements of the ascending chain described above, i.e. $X_{s_0} \subset X_{s_1s_0} \subset X_{s_0s_1s_0} \subset \cdots \subset \Omega SU(2)$. For more details, see \cite{mitchell}.\\

For a reduced word $w = s_{i_1} \cdots s_{i_n}$ define $S_w = \displaystyle\prod\limits_{j=1}^n e^{s_{i_1} \cdots s_{i_{j-1}}(\alpha_{i_j})}$. Denote $S_{w_n}$ by $S_n$, and similarly denote $R_{w_n}, X_{w_n}$ by $R_n, X_n$ respectively. 

\begin{definition}
We define the following restricted partitions: \\
$P_{n, k}(m):= $ the number of integer partitions of $m$ into at most $n$ parts, all of which are at most $k$. \\
$q_{n, k}(m) := $ the number of integer partitions of $m$ into exactly $n$ parts, all of which are at most $k$.
\end{definition}
\begin{example}
$P_{3, 2}(4) = 2$ since we have $2 + 2$, $2 + 1 + 1$,  $q_{3, 2}(4) = 1$ since we only have the partition $2 + 1 + 1$. 
\end{example}
We also denote the number of all integer partitions of $m$ by $P(m)$. 
\\

We take a moment here to give a reminder on the roots of the affine Lie algebra $\hat{\mathfrak{sl}}_2$. 
In contrast to semisimple Lie algebras, affine Lie algebras have both real and imaginary roots (roots which have nonpositive length). Let $\delta = \alpha_0 + \alpha_1$ denote the smallest positive imaginary root of $\hat{\mathfrak{sl}}_2$. The real roots of $\hat{\mathfrak{sl}}_2$ are given by the set $\triangle_{re} = \{\pm\alpha_1 + n\delta \hspace{0.1cm} | \hspace{0.1cm} n \in \mathbb{Z}\}$ and the positive real roots form the set $\triangle_{re} = \{\alpha_1\} \cup \{\pm\alpha_1 + n\delta \hspace{0.1cm} | \hspace{0.1cm} n \in \mathbb{N} \}$. For more details on roots of affine Lie algebras, see \cite{kac}.\\

\begin{theorem} \label{v1}
The rational functions in the localization formula for the affine Grassmannian $\Omega G$ are given by 
\begin{equation}
   R_m(\Omega G) = \displaystyle\frac{(-1)^{m}f(w_m)(1 - e^{w_m \cdot \alpha_1}) \displaystyle\sum\limits_{n = 0}^{\infty} P(n)e^{n\delta} }{\displaystyle\prod\limits_{\alpha \in \triangle^{+}_{re}} (1 - e^{\alpha}) }
\end{equation}
where $\triangle^{+}_{re}$ denotes the \textbf{positive real} roots of the affine Lie algebra $\hat{\mathfrak{sl}}_2$. 
\end{theorem}

\begin{theorem} (Second version)\label{v2} 
The rational functions in the localization formula for the affine Grassmannian $\Omega G$ are given by 
\begin{equation}
   R_m(\Omega G) = \displaystyle\frac{(-1)^{m}f(w_m)(1 - e^{w_m \cdot \alpha_1})  }{\displaystyle\prod\limits_{\alpha \in \triangle^{+}} (1 - e^{\alpha}) }
\end{equation}
where $\triangle^{+}$ denotes the \textbf{positive} roots of the affine Lie algebra $\hat{\mathfrak{sl}}_2$. 
\end{theorem}

The sum in the numerator in Theorem \ref{v1} is simply the infinite product $\left( \displaystyle\prod\limits_{n=0}^{\infty} (1 - e^{n\delta}) \right)^{-1}$, and each factor encodes an imaginary root, complementing the infinite product over all positive real roots in Theorem \ref{v1}. \\

Let us do a few sample computations. \\
$R_1(X_3) = \displaystyle\frac{(-1)S_1(1 - e^{w_1 \cdot \alpha_1})(1 + e^{\delta} + e^{2\delta})}
{ ( 1 - e^{\alpha_1})( 1 - e^{\alpha_0})( 1 - e^{\alpha_0 + \delta})( 1 - e^{\alpha_0 + 2\delta}) }$ \\

$R_2(X_4) = \displaystyle\frac{S_2(1 - e^{w_2 \cdot \alpha_1})(1 + e^{\delta} + e^{2\delta} + e^{3\delta})}
{ ( 1 - e^{\alpha_0})( 1 - e^{\alpha_1})( 1 - e^{\alpha_1 + \delta})( 1 - e^{\alpha_1 + 2\delta})( 1 - e^{\alpha_1 + 3\delta}) }$ \\

$R_0(X_4) = \displaystyle\frac{(1 - e^{\alpha_1})(1 + e^{\delta} + e^{2\delta})(1 + e^{2\delta})}
{ ( 1 - e^{\alpha_0})( 1 - e^{\alpha_0 + \delta})( 1 - e^{\alpha_1})( 1 - e^{\alpha_1 + \delta}) }$ \\

$R_{0}(X_3) = \displaystyle\frac{(1 + e^{\delta} + e^{2\delta})}
{ ( 1 - e^{\alpha_0})( 1 - e^{\alpha_0 + \delta})( 1 - e^{\alpha_1 + \delta}) }$

To prove the theorem we first prove a finite-dimensional version, let $\overline{m}$ denote $m (mod 2)$ for any integer $m$.

\begin{lemma} 
\label{fin}
The rational functions in the localization formula for the Schubert varieties are given by 
\begin{align*}
R_m(X_n) =
   \ddfrac{(-1)^{m}f(w_m)(1 - e^{w_m \cdot \alpha_1}) \left( \displaystyle\sum\limits_{j=0}^{ \left( n - \floor*{\frac{n-m}{2}} \right)\floor*{\frac{n-m}{2}} } P_{\floor*{\frac{n-m}{2}} , n - \floor*{\frac{n-m}{2}} }(j) e^{j\delta} \right) }
   {
   \left(\displaystyle\prod\limits_{i = 0}^{m + \floor*{\frac{n-m}{2}}} 1 - e^{\alpha_{\overline{m+1}} + j\delta} \right) 
   \left( \displaystyle\prod\limits_{j = 0}^{\ceil*{\frac{n-m}{2}} - 1} 1 - e^{\alpha_{\overline{m}} + j\delta} 
   \right)
   }
\end{align*}
\end{lemma}

\textbf{Remark:} It is possible to write the sum in the numerator with $j$ ranging from 0 to $\infty$. However, since we are dealing with restricted partitions there is an effective limit $\ell$ above which $P_{a, b}(N) = 0$ for all $N > \ell$. This limit is given exactly by $\ell = ab$, so we have written that limit for the sum in this lemma. 

Before we prove \ref{fin} we will need the following identity on restricted partitions:

\begin{proposition} \label{part}
$P_{a, b}(j) + P_{a+1, b-1}(j - a - 1) = P_{a+1, b}(j)$
\end{proposition}
\begin{proof}
Notice that by definition, $$P_{a+1, b}(j) = \displaystyle\sum\limits_{n=0}^{a+1}q_{n, b}(j) = \displaystyle\sum\limits_{n=0}^{a}q_{n, b}(j) + q_{a+1, b}(j) = P_{a, b}(j) + q_{a+1, b}(j)$$ 
therefore it suffices to prove that $P_{a+1, b-1}(j - a - 1) = q_{a+1, b}(j)$. First we establish some bounds on $j$: If $j > (a+1)b$, then both sides equal 0. Also if $j < a+1$, then both sides equal 0. So we assume that $a+1 \le j \le (a+1)b$. \\
Then $P_{a+1, b-1}(j - a - 1)$ counts the number of Young diagrams with at most $a+1$ rows and at most $b-1$ columns, with $j-a-1 = j-(a+1)$ total blocks. We call this set of Young diagrams $Y_{a+1, b-1}(j-(a+1))$. On the other hand $q_{a+1, b}(j)$ counts the number of Young diagrams with exactly $a+1$ rows and at most $b-1$ columns, with $j$ total blocks, call this set of Young diagrams $y_{a+1, b-1}(j)$. Then we have a map
\begin{align*}
    Y_{a+1, b-1}(j-(a+1)) &\rightarrow y_{a+1, b-1}(j) \\
    (\lambda_1, \cdots, \lambda_{a+1}) &\mapsto (\lambda_1 + 1, \cdots, \lambda_{a+1}+1)
\end{align*}
which is in fact a bijection. Therefore $P_{a+1, b-1}(j - a - 1) = q_{a+1, b}(j)$. 
\end{proof}

\begin{proposition} \label{partcor}
$P_{a+1, b-1}(j) + P_{a, b}(j - b) = P_{a+1, b}(j)$
\end{proposition}
\begin{proof}
By the correspondence to Young diagrams, we see that $P_{a, b}(j) = P_{b , a}(j)$ for all $a, b, j$ just by conjugation. Then the identity follows from Proposition \ref{part} by setting $a = b-1, b = a+1$. 
\end{proof}

Finally we need the following elementary identities:
\begin{proposition} \label{elem}
If $k, m$ are integers with $k > m$ and $k - m$ odd, then 
\begin{itemize}
\item $\left( k - \floor*{\frac{k-m}{2}} \right)\floor*{\frac{k-m}{2}} + (m + \floor*{\frac{k-m}{2}} + 1) = \left( k - \floor*{\frac{k-m}{2}} \right)(\floor*{\frac{k-m}{2}} + 1)$
\item $\left( k - \floor*{\frac{k-(m-1)}{2} } \right)\floor*{\frac{k-(m-1)}{2}} + \ceil*{\frac{k-m}{2}} = \left( k - \floor*{\frac{k-m}{2}} \right)(\floor*{\frac{k-m}{2}} + 1)$
\item $m + \floor*{\frac{k-m}{2}} + 1 = k - \floor*{\frac{k-m}{2}}$
\end{itemize}
\end{proposition}
\begin{proof}
An elementary calculation. 
\end{proof}

\noindent
\textbf{\underline{Proof of lemma \ref{fin}:}} \\
The proof of the lemma is by induction on $n$. \\
For $n = 1$, $X_{s_0} = M_{s_0} = 
\begin{pmatrix}
a & -\bar{b}t \\
bt^{-1} & \bar{a}
\end{pmatrix}$ where $a, b \in \mathbb{C}$. This space is diffeomorphic to 
 $\mathbb{CP}^1$ and the only fixed points of $X_1$ are the north and south poles whose moment map images are $id$ and $s_0$ respectively (i.e. the correspond to the Weyl group elements $id$ and $s_0$). Using formula (\ref{bwv}) we get that $R_0(X_{s_0}) = R_{id}(X_{s_0}) = \frac{1}{1 - e^{\alpha_0}}$ and $R_1(X_{s_0}) = R_{s_0}(X_{s_0}) = \frac{1}{1 - e^{-\alpha_0}} = \frac{-e^{a_0}}{1 - e^{\alpha_0}}$. So the lemma is true for $n = 1$. \\
Now assume that the lemma is true for $n = k$, now we wish to compute the rational functions $R_m(X_{k+1})$ for all $m \le k+1$. Notice that $w_{k+1} = s_{\alpha} \cdot w_k$ for some simple root $\alpha$, and we recall that $R_m(X_{k+1})$ is a sum of the rational functions $\displaystyle\sum\limits_{mW = m^{'}W}\overline{b_{k+1, m^{'}}}$. Notice that if $w = s_i \cdot w^{'}$ then we have the identity $b_{w, v} = (1 - e^{\alpha_i})^{-1} \left(b_{w^{'}, v} - e^{-\alpha_i} s_i \cdot b_{w^{'}, s_i \cdot v} \right)$, since the subwords of $w$ that equal $v$ are either: \textbf{(1)} $id$ concatenated with the subwords of $w^{'}$ that form $v$, or \textbf{(2)} $s_i$ concatenated with the subwords of $w^{'}$ that form $s_i \cdot v$. And because $s_i$ turns $\alpha_i$ into a negative root, we ``polarize" the second term with a factor of $-e^{\alpha_i}$. Therefore putting together the last two formulas we get $R_m(X_{k+1}) = (1 - e^{\alpha_i})^{-1} \left(R_m(X_k) - e^{\alpha_i} s_i \cdot R_{s_i \cdot w_m}(X_k) \right)$. \\

Notice that when $m = 0$ and $k+1$ is even (i.e. the fixed point is the identity coset and $s_i = s_1$), then $s_1$ and $e$ belong to the same coset, so in this case $R_0(X_{k+1}) = (1 - e^{\alpha_1})^{-1}(1 - e^{\alpha_1} s_1)R_0(X_k)$, which amounts to applying the Demazure operator $D_{s_1}$ to $R_0(X_k)$. \\

There are two cases, either, $s_i \cdot w_m = w_{m-1}$ or $s_i \cdot w_m = w_{m+1}$. \\
Case 1: $s_i \cdot w_m = w_{m-1}$ which implies that $k-m$ is odd \\
In this case we have

\begin{enumerate}
    \item $\floor*{\frac{k-m}{2}} + 1 = \ceil*{\frac{k-m}{2}} = \floor*{\frac{k-(m-1)}{2}}= \ceil*{\frac{k-(m-1)}{2}}$\\
    \item $s_i = s_{\alpha_{\overline{m+1
          }}}$
    \item $s_i \cdot (\alpha_k + j\delta) = s_i \cdot \alpha_k + j\delta$
    \item $s_i \cdot f(w_{m-1}) = -e^{-\alpha_i} \cdot f(w_m)$ by definition
\end{enumerate}

Now by the induction hypothesis we can compute $R_m(X_{k+1})$ readily as 
\begin{align*}
R_m(X_{k+1}) = &(1 - e^{\alpha_i})^{-1} \left(R_m(X_k) - e^{-\alpha_i} s_i \cdot R_{s_i \cdot w_m}(X_k) \right) \\
= &(1 - e^{\alpha_i})^{-1} \left(R_m(X_k) - e^{-\alpha_i} s_i \cdot R_{m-1}(X_k) \right)
\end{align*}
Before carrying out the induction step, we simplify notation by letting 
$$\boldsymbol{A} = \displaystyle\sum\limits_{j=0}^{\left( k - \floor*{\frac{k-m}{2}} \right)\floor*{\frac{k-m}{2}}} P_{\floor*{\frac{k-m}{2}}, k - \floor*{\frac{k-m}{2}}}(j)e^{j\delta}$$ and
   $$\boldsymbol{B} = \displaystyle\sum\limits_{j=0}^{\left( k - \floor*{\frac{k-(m-1)}{2}} \right)\floor*{\frac{k-(m-1)}{2}}} P_{\floor*{\frac{k-(m-1)}{2}}, k - \floor*{\frac{k-(m-1)}{2}}}(j) e^{j\delta}$$
We now give a detailed computation of $R_m(X_k) - e^{-\alpha_i} s_i \cdot R_{s_i \cdot w_m}(X_k)$:

\begin{align*}
    &R_m(X_k) - e^{-\alpha_i} s_i \cdot R_{m-1}(X_k) \\
    = \hspace{1.0cm} &
    \ddfrac{(-1)^{m}f(w_m)(1 - e^{w_m \cdot \alpha_1})   \boldsymbol{A} }
   {
   \left(\displaystyle\prod\limits_{j = 0}^{m + \floor*{\frac{k-m}{2}}} 1 - e^{\alpha_{\overline{m+1}} + j\delta} \right) 
   \left( \displaystyle\prod\limits_{j = 0}^{\ceil*{\frac{k-m}{2}} - 1} 1 - e^{\alpha_{\overline{m}} + j\delta} 
   \right)
   }  \\
   - \hspace{0.5cm} &  e^{\alpha_i} s_i \cdot \ddfrac{(-1)^{m-1}f(w_{m-1})(1 - e^{w_{m-1} \cdot \alpha_1}) \boldsymbol{B} }
   {
   \left(\displaystyle\prod\limits_{j = 0}^{m + \floor*{\frac{k-m}{2}}} 1 - e^{\alpha_{\overline{(m-1)+1}} + j\delta} \right) 
   \left( \displaystyle\prod\limits_{j = 0}^{\ceil*{\frac{k-m}{2}} - 1} 1 - e^{\alpha_{\overline{m-1}} + j\delta} 
   \right)
   } \\
   = \hspace{1.0cm} &
    \ddfrac{(-1)^{m}f(w_m)(1 - e^{w_m \cdot \alpha_1}) \boldsymbol{A} }
   {
   \left(\displaystyle\prod\limits_{j = 0}^{m + \floor*{\frac{k-m}{2}}} 1 - e^{\alpha_{\overline{m+1}} + j\delta} \right) 
   \left( \displaystyle\prod\limits_{j = 0}^{\ceil*{\frac{k-m}{2}} - 1} 1 - e^{\alpha_{\overline{m}} + j\delta} 
   \right)
   }  \\
   - \hspace{0.5cm} & e^{\alpha_i} \ddfrac{s_i \cdot (-1)^{m-1}f(w_{m-1})(1 - e^{w_{m-1} \cdot \alpha_1}) 
   \boldsymbol{B}  }
   {
   \left(\displaystyle\prod\limits_{j = 0}^{m + \floor*{\frac{k-m}{2}}} 1 - e^{s_i \cdot \alpha_{\overline{m}} + j\delta} \right) 
   \left( \displaystyle\prod\limits_{j = 0}^{\ceil*{\frac{k-m}{2}} - 1} 1 - e^{s_i \cdot \alpha_{\overline{m+1}} + j\delta} 
   \right)
   } \\
      \intertext{(by definition, $s_i$ acting on $f(w_{m-1})$ turns it into $f(w_m)$. The reflection $s_i$ also acts on the weights $e^{\lambda}$ and transforms them by reflecting across the hyperplane associated to the root $\alpha_i$)}
   = \hspace{1.0cm} &
    \ddfrac{(-1)^{m}f(w_m)(1 - e^{w_m \cdot \alpha_1}) \boldsymbol{A} }
   {
   \left(\displaystyle\prod\limits_{j = 0}^{m + \floor*{\frac{k-m}{2}}} 1 - e^{\alpha_{\overline{m+1}} + j\delta} \right) 
   \left( \displaystyle\prod\limits_{j = 0}^{\ceil*{\frac{k-m}{2}} - 1} 1 - e^{\alpha_{\overline{m}} + j\delta} 
   \right)
   }  \\
   + \hspace{0.5cm} &  \ddfrac{(-1)^{m}f(w_m)(1 - e^{w_m \cdot \alpha_1}) \boldsymbol{B} }
   {
   \left(\displaystyle\prod\limits_{j = 1}^{m + \floor*{\frac{k-m}{2}} + 1} 1 - e^{\alpha_{\overline{m+1}}  + j\delta} \right) 
   \left( \displaystyle\prod\limits_{j = -1}^{\ceil*{\frac{k-m}{2}} - 2} 1 - e^{\alpha_{\overline{m}} + j\delta} 
   \right)
   }  \\
   \intertext{(renumbering the indices to start at 1 and -1)}
   \intertext{Now notice that $\alpha_{\overline{m}} - \delta = -\alpha_{\overline{m+1}}$, so}
   = \hspace{1.0cm} &
    \ddfrac{(-1)^{m}f(w_m)(1 - e^{w_m \cdot \alpha_1}) \boldsymbol{A} }
   {
   \left(\displaystyle\prod\limits_{j = 0}^{m + \floor*{\frac{k-m}{2}}} 1 - e^{\alpha_{\overline{m+1}} + j\delta} \right) 
   \left( \displaystyle\prod\limits_{j = 0}^{\ceil*{\frac{k-m}{2}} - 1} 1 - e^{\alpha_{\overline{m}} + j\delta} 
   \right)
   }  \\
   + \hspace{0.5cm} &  \ddfrac{(-1)^{m}f(w_m)(1 - e^{w_m \cdot \alpha_1}) \boldsymbol{B} }
   {
   \left(\displaystyle\prod\limits_{j = 1}^{m + \floor*{\frac{k-m}{2}} + 1} 1 - e^{\alpha_{\overline{m+1}}  + j\delta} \right) 
   \left( \displaystyle\prod\limits_{j = 0}^{\ceil*{\frac{k-m}{2}} - 2} 1 - e^{\alpha_{\overline{m}} + j\delta} 
   \right)(1 - e^{-\alpha_{\overline{m+1}}})
   }  \\
   = \hspace{1.0cm} &
    \ddfrac{(-1)^{m}f(w_m)(1 - e^{w_m \cdot \alpha_1}) \boldsymbol{A} }
   {
   \left(\displaystyle\prod\limits_{j = 0}^{m + \floor*{\frac{k-m}{2}}} 1 - e^{\alpha_{\overline{m+1}} + j\delta} \right) 
   \left( \displaystyle\prod\limits_{j = 0}^{\ceil*{\frac{k-m}{2}} - 1} 1 - e^{\alpha_{\overline{m}} + j\delta} 
   \right)
   }  \\
   + \hspace{0.5cm} &  \ddfrac{(-1)^{m}f(w_m)(1 - e^{w_m \cdot \alpha_1}) \boldsymbol{B}(-e^{\alpha_{\overline{m}}}) }
   {
   \left(\displaystyle\prod\limits_{j = 0}^{m + \floor*{\frac{k-m}{2}} + 1} 1 - e^{\alpha_{\overline{m+1}}  + j\delta} \right) 
   \left( \displaystyle\prod\limits_{j = 0}^{\ceil*{\frac{k-m}{2}} - 2} 1 - e^{\alpha_{\overline{m}} + j\delta} 
   \right)
   }  \\
   \intertext{(in the previous step we polarized the negative weight $1 - e^{-\alpha_{\overline{m+1}}}$ into $(1 - e^{\alpha_{\overline{m+1}}})(-e^{\alpha_{\overline{m+1}}})^{-1}$ and absorbed $1 - e^{\alpha_{\overline{m+1}}}$ into the bottom right sum)}
   = \hspace{1.0cm} &
   \ddfrac{(-1)^{m}f(w_m)(1 - e^{w_m \cdot \alpha_1}) \left( 
   (1 - e^{\alpha_{\overline{m+1}} + (m + \floor*{\frac{k-m}{2}} + 1)\delta})
   \boldsymbol{A} + 
   (-e^{\alpha_{\overline{m+1}}})(1 - e^{\alpha_{\overline{m}} + (\ceil*{\frac{k-m}{2}} - 1)\delta})
   \boldsymbol{B} \right) }
   {
   \left(\displaystyle\prod\limits_{j = 0}^{m + \floor*{\frac{k-m}{2}} + 1} 1 - e^{\alpha_{\overline{m+1}}  + j\delta} \right) 
   \left( \displaystyle\prod\limits_{j = 0}^{\ceil*{\frac{k-m}{2}} - 1} 1 - e^{\alpha_{\overline{m}} + j\delta} 
   \right)
   }
\end{align*}
\vspace{1cm}

We calculate the numerator of the last expression above. We set $a = \floor*{\frac{k-m}{2}}$ and $b = k - \floor*{\frac{k-m}{2}}$ for more clarity in the following formulae. 
\begin{align*}
    &\left( (1 - e^{\alpha_{\overline{m+1}} + (m + \floor*{\frac{k-m}{2}} + 1)\delta})
   \displaystyle\sum\limits_{j=0}^{\left( k - \floor*{\frac{k-m}{2}} \right)\floor*{\frac{k-m}{2}}} P_{\floor*{\frac{k-m}{2}}, k - \floor*{\frac{k-m}{2}}}(j)e^{j\delta} \right)\\
  -&\left( e^{\alpha_{\overline{m+1}}}(1 - e^{\alpha_{\overline{m}} + (\ceil*{\frac{k-m}{2}} - 1)\delta})
   \displaystyle\sum\limits_{j=0}^{\floor*{\frac{k-(m-1)}{2}}} P_{\floor*{\frac{k-(m-1)}{2}}, k - \floor*{\frac{k-(m-1)}{2}}}(j) e^{j\delta} \right) \\
   =&\left( (1 - e^{\alpha_{\overline{m+1}} + (m + a + 1)\delta})
   \displaystyle\sum\limits_{j=0}^{ab} P_{a, b}(j)e^{j\delta} \right)
  -\left( e^{\alpha_{\overline{m+1}}}(1 - e^{\alpha_{\overline{m}} + a\delta})
   \displaystyle\sum\limits_{j=0}^{a + 1} P_{a + 1, b - 1}(j) e^{j\delta} \right) \\
   =& \displaystyle\sum\limits_{j=0}^{ab} P_{a, b}(j)e^{j\delta} 
   + e^{(a +1)\delta}\displaystyle\sum\limits_{j=0}^{\left( b - 1 \right)(a+1)} P_{a+1, b-1}(j) e^{j\delta} 
   -e^{\alpha_{\overline{m+1}}} \left\{ \displaystyle\sum\limits_{j=0}^{(a+1)(b-1)}  P_{a+1, b-1}(j)e^{j\delta}    + e^{(m + a + 1)\delta} \displaystyle\sum\limits_{j=0}^{ab}P_{a, b} \left(j \right) 
   e^{j\delta} \right\}
   \\
   \intertext{now we re-index the sums and partitions: In the second sum we reindex $j \leftarrow j + b+1$ and in the fourth sum we reindex $j \leftarrow j + (m + a + 1)$. Then we can change the lower limits in the (reindexed) sums to ~0 because the partitions of integers below $a+1$ and $(m + a + 1)$ will both be 0}
   =& \displaystyle\sum\limits_{j=0}^{ab} P_{a, b}(j)e^{j\delta} 
   + \displaystyle\sum\limits_{j=0}^{b(a+1)} P_{a+1, b-1} \left(j - a - 1 \right) e^{j\delta} \\
   &-e^{\alpha_{\overline{m+1}}} \left\{ \displaystyle\sum\limits_{j=0}^{(a+1)(b-1)}  P_{a+1, b-1}(j)e^{j\delta} +\displaystyle\sum\limits_{j=0}^{ab+m+a+1} P_{a, b} \left(j - (m + a + 1) \right) 
   e^{j\delta} \right\}
   \\
   \intertext{using Proposition \ref{elem} we can rewrite the partitions and upper limits of the sums as}
   =& \displaystyle\sum\limits_{j=0}^{ b (a + 1)} \left( P_{a, b}(j)  + P_{a+1, b-1} \left(j - a - 1 \right) \right)
   e^{j\delta} 
   -e^{\alpha_{\overline{m+1}}} \cdot \left\{ \displaystyle\sum\limits_{j=0}^{b(a+1)} \left( P_{a+1, b-1}(j)  + P_{a, b} (j - b ) \right)
   e^{j\delta} \right\}
\end{align*}

We see that by using Proposition \ref{part} and Corollary \ref{partcor}, the partitions in both sums become $P_{a + 1, b}(j)$. Finally after factoring out the $(1 - e^{\alpha_{\overline{m+1}}}) = (1 - e^{\alpha_{i}})$ term we arrive at the correct formula for $X_{k+1}$, thereby proving the induction step. \\

The other case is similar. \qed \\

Theorems \ref{v1} and \ref{v2} follow from taking the limit as $k \rightarrow \infty$ of the expression in Lemma \ref{fin}. The restricted partitions become partitions outright. We get the rational functions $R_w(\Omega SU(2)): T \times S^1 \rightarrow \mathbb{C}$ for each fixed point of $\Omega SU(2)$. \\

\textbf{Remark:} Strictly speaking, we should have taken the complex conjugate of the quantities in Theorems \ref{v1} and \ref{v2}. This is because the character of the positive energy representation was equal to the inverse limit of the \textbf{dual} representations given by the pullback bundles on the Bott-Samelson manifolds (see Theorem \ref{invlimit}). However, we note that this is mostly a matter of convention (using lowest weights vs. highest weights). The standard convention is with highest weights, as in Kumar \cite{kum}. However, Pressley and Segal use the lowest weight convention in their book \textit{Loop Groups} (see \cite{ps}). 

\section{Demazure modules}
A consequence of Lemma \ref{fin} is that it allows us to write an effective character formula for Demazure modules associated to the Schubert varieties. Previously, these characters were only given in terms of iterated Demazure operators. In this section we let $\mathfrak{g}$ be a Kac-Moody algebra and $\mathfrak{b}$ a Borel subalgebra.

\begin{definition}
Let $\lambda$ be a dominant weight, and let $V(\lambda)$ be the irreducible representation with highest-weight $\lambda$. For any $w \in \overline{W}$, form $E_w(\lambda) = \mathfrak{b} \cdot v_{w(\lambda)}$; it is a submodule of $V(\lambda)$ called the Demazure submodule of $V(\lambda)$ associated to $w$.
\end{definition}

The Demazure submodule of $V(\lambda)$ associated to $w$ is isomorphic to the restriction of the prequantum line bundle $L_{\lambda}$ to the Schubert variety $X_w$ (see Kumar \cite{kum}). Therefore Lemma \ref{fin} allows us to effectively compute the character of $E_w(\lambda)$ for $\lambda = k\overline{\omega_0}$ and for any $w \in \overline{W}$. \\

\textbf{Remark:} Previously, characters of Demazure modules were computed using the Demazure character formula. This formula allows one to write the character $ch \hspace{0.1cm}E_w$ as iterated applications of certain \textbf{Demazure operators} to the weight $e^{\lambda}$, i.e.  $ch \hspace{0.1cm} E_w = D_{s_{i_1}} \cdots D_{s_{i_n}} e^{\lambda}$ where $w = s_{i_1} \cdots s_{i_n}$ is a decomposition of the reduced word $w$. These Demazure operators act on the group algebra $A[T]$ (see Section \ref{nil-Hecke}). For a simple reflection $s_i$, the Demazure operator is defined by 
$$ D_{s_i}(e^{\lambda}) = \frac{e^{\lambda} - e^{s_i \lambda - \alpha_i} }{1 - e^{-\alpha_i}}.$$

\section{A Kostant Multiplicity function}
Each factor $(1 - e^{\alpha})^{-1}$ can be written as $\displaystyle\sum\limits_{n=0}^{\infty} e^{n\alpha}$, so we can compute the multiplicities of the character in the last section as a sum of partition functions, defined below
\begin{definition}
The partition function of a weight $\mu \in Lie(\mathcal{T})^{*}$ (see Section 6 for this notation) is defined as $N(\mu) = $\# of solutions of the equation $\displaystyle\sum\limits_{\alpha \in \triangle^+} n_{\alpha}\alpha = \mu$ where each $n_{\alpha} \in \mathbb{N} \cup \{0\}$. \\
Also define 
$N_{\beta}(\mu) = $\# of solutions of the equation $\displaystyle\sum\limits_{\alpha \in \triangle^+ \setminus \{\beta\}} n_{\alpha}\alpha = \mu$ where each $n_{\alpha} \in \mathbb{N} \cup \{0\}$. \\
\end{definition}

For a reduced word $w = s_{i_1} \cdots s_{i_n}$ we also define $f(w) = \displaystyle\sum\limits_{j=1}^n s_{i_1} \cdots s_{i_{j-1}} (\alpha_{i_j})$.  \\

Then by reading off the rational functions derived in the last section, we obtain the following Kostant multiplicity formula for $\Omega SU(2)$:
\begin{proposition}
Let $\lambda = k\overline{w}_0, k \in \mathbb{Z}_{>0}$, then the multiplicity of the weight $\alpha$ in the irreducible representation $L_{\lambda}$ is given by
\begin{align*} 
    m(\alpha, L_{\lambda}) &= \displaystyle\sum\limits_{m = 0}^{\infty} (-1)^m \left\{ 
    N(\alpha - (w_m \cdot \lambda + f(w_m))) - N(\alpha - (w_m \cdot (\lambda + \alpha_1) + f(w_m) ))
    \right\} \\
    &= \displaystyle\sum\limits_{m = 0}^{\infty} (-1)^m N_{w_m \cdot \alpha_1}(\alpha - (w_m \cdot \lambda + f(w_m)))
\end{align*}
\end{proposition}\label{kostant}

An immediate corollary is the following identity:

\begin{corollary}
\begin{equation}
 N(\alpha - (w_m \cdot \lambda + f(w_m))) - N(\alpha - (w_m \cdot (\lambda + \alpha_1) + f(w_m) ))
 =N_{w_m \cdot \alpha_1}(\alpha - (w_m \cdot \lambda + f(w_m))) 
\end{equation}
\end{corollary}

\textbf{Remark:} The sum given above will always be finite. This is because only finitely many terms will be partitions of positive weights. \\

\end{document}